\documentclass[11pt]{amsart}
\usepackage{amssymb,amscd,amsthm,verbatim}
\usepackage{amsbsy}
\usepackage{latexsym}
\usepackage[all,cmtip]{xy}
\usepackage[mathscr]{euscript}

\usepackage{etex}
\usepackage{graphicx}
\usepackage{pictex}
\usepackage{epstopdf}
\usepackage{color}

\usepackage{bbm}

\date{\today}

\newcommand{\Z}{{\mathbb Z}}
\newcommand{\R}{{\mathbb R}}

\newcommand{\C}{{\mathbb C}}

\newcommand{\T}{{\mathbb T}}

\newcommand{\N}{{\mathbb N}}

\renewcommand{\Re}{{\mathrm{Re} \,}}

\newtheorem{theorem}{Theorem}[section]
\newtheorem{lemma}[theorem]{Lemma}
\newtheorem{prop}[theorem]{Proposition}
\newtheorem{coro}[theorem]{Corollary}

\theoremstyle{definition}

\newtheorem{remark}[theorem]{Remark}
\newtheorem{remarks}[theorem]{Remarks}
\theoremstyle{plain}

\allowdisplaybreaks \numberwithin{equation}{section}

\newcommand{\set}[1]{\left\{#1\right\}}

\begin{document}

\title[Quasi-Periodic Schr\"odinger Operators and Singular Spectrum]{Absence of Absolutely Continuous Spectrum for Generic Quasi-Periodic Schr\"odinger Operators on the Real Line}

\author[D.\ Damanik]{David Damanik}
\address{Department of Mathematics, Rice University, Houston, TX~77005, USA}
\email{damanik@rice.edu}
\thanks{D.D.\ was supported in part by NSF grant DMS--1700131.}

\author[D.\ Lenz]{Daniel Lenz}
\address{Institute for Mathematics, Friedrich-Schiller University, Jena, 07743 Jena}
\email{daniel.lenz@uni-jena.de}

\begin{abstract}
We show that a generic quasi-periodic Schr\"odinger operator in $L^2(\R)$ has purely singular spectrum. That is, for any minimal translation flow on a finite-dimensional torus, there is a residual set of continuous sampling functions such that for each of these sampling functions, the Schr\"odinger operator with the resulting potential has empty absolutely continuous spectrum.
\end{abstract}

\maketitle

\section{Introduction}

In this paper we consider Schr\"odinger operators
\begin{equation}\label{e.operator}
[H \psi](x) = -\psi''(x) + V(x) \psi(x)
\end{equation}
in $L^2(\R)$ with quasi-periodic potentials
\begin{equation}\label{e.potential}
V(x) = f( \omega + x \alpha ).
\end{equation}
Here, $\omega, \alpha \in \T^d = \R^d/\Z^d$ for some $d \in \Z_+$, $f \in C(\T^d)$ real-valued, and $x \in \R$. The case where $V$ is periodic is classical and well understood, and hence we will primarily focus on the aperiodic case. This necessarily means that $d \ge 2$ and it also places some restrictions on $\alpha$ and $f$. We will assume that $\alpha$ is such that the translation flow in question is minimal (i.e., all orbits are dense) to ensure that the torus dimension $d$ is chosen appropriately, and moreover $f$ needs to be non-constant to avoid periodicity.

The spectral properties of operator of the form \eqref{e.operator} with potentials of the form \eqref{e.potential} have been studied intensively since the 1980's, with many major advances occurring in the past two decades. Much of this work has been reviewed in several recent survey papers, including \cite{D07, D17, J07, JM17, K15}. We should point out, however, that some of these survey papers discuss the discrete analogs of these operators, which act in $\ell^2(\Z)$ as
$$
[H^{(d)} \psi](n) = \psi(n+1) + \psi(n-1) + V^{(d)}(n) \psi(n)
$$
with
$$
V^{(d)}(n) = f( \omega + n \alpha ),
$$
but many results exist in both settings.

There are, of course, some notable exceptions. One of the most important exceptions is that Avila's global theory for discrete one-frequency quasi-periodic Schr\"odinger operators with analytic sampling functions \cite{A15} does not yet have a continuum counterpart. In this paper we will address another result, which is known in the discrete case, but whose continuum counterpart is desirable to have because of recent progress on the Deift conjecture, which makes a connection with continuum quasi-periodic Schr\"odinger operators.

The Deift conjecture \cite{D08, De17} states that the KdV equation with almost periodic initial data admits global solutions that are almost periodic in (space and) time. The conjecture has been proved under suitable assumptions \cite{BDGL18+, EVY18+}. These results, and really their proofs, need that the initial data, when considered as potentials, give rise to Schr\"odinger operators with absolutely continuous spectrum. It was therefore pointed out in \cite{BDGL18+} that the assumptions will likely fail generically in a suitable sense.

Concretely, the abstract sufficient conditions for the Deift conjecture to hold have been verified for suitable classes of quasi-periodic functions of the form \eqref{e.potential}; see \cite{BDGL18+}. On the other hand, for \emph{discrete} quasi-periodic Schr\"odinger operators, it is known \cite{AD05} that a generic quasi-periodic potential will give rise to a Schr\"odinger operator with \emph{empty} absolutely continuous spectrum. One should therefore expect that also in the continuum case, which is the one relevant to the study of the KdV equation and the Deift conjecture, the absolutely continuous spectrum will be empty for a generic quasi-periodic potential.

\medskip

The purpose of this paper is to prove this statement:

\begin{theorem}\label{t.main1}
Given $d \ge 2$ and a minimal translation flow on $\T^d$, $\R \ni x \mapsto \omega + x \alpha \in \T^d$, there is a dense $G_\delta$-set $\mathcal{S} \subseteq C(\T^d)$ such that for every $f \in \mathcal{S}$, the Schr\"odinger operator in $L^2(\R)$ with potential $V(x) = f(\omega + x \alpha)$ has purely singular spectrum.
\end{theorem}

\begin{remarks}
(a) The minimality of the flow is a property of $\alpha \in \T^d$, and the result holds for any such fixed $\alpha$. The set $\mathcal{S}$ will then depend on the choice of $\alpha$.\\[1mm]
(b) There is no quantifier on $\omega \in \T^d$ in the statement of the result, even though the potential $V$ depends on it. This is due to the constancy of the absolutely continuous spectrum in $\omega$, which is a result of Last and Simon \cite[Theorem~1.5]{LS99}.\\[1mm]
(c) This result shows that there is a generic obstruction to an extension of the BDGL approach \cite{BDGL18+} or the EVY approach \cite{EVY18+} to the Deift conjecture \cite{D08, De17}.

\end{remarks}

One can also consider one-parameter families of potentials and operators by varying the coupling constant:

\begin{theorem}\label{t.main2}
Given $d \ge 2$ and a minimal translation flow on $\T^d$, $\R \ni x \mapsto \omega + x \alpha \in \T^d$, there is a dense $G_\delta$-set $\mathcal{S} \subseteq C(\T^d)$ such that for every $f \in \mathcal{S}$ and Lebesgue almost every $\lambda > 0$, the Schr\"odinger operator in $L^2(\R)$ with potential $V(x) = \lambda f(\omega + x \alpha)$ has purely singular spectrum.
\end{theorem}

\section{Preliminaries}

\subsection{Discontinuous Periodic Functions Having Limit-Periodic Limits}

Recall that a bounded uniformly continuous function on $\R$ is
called \emph{almost periodic} if for any $\varepsilon >0$ the set of
$t\in\R$ with $\|f - f(\cdot -t)\|_\infty <\varepsilon$ is
relatively dense. A bounded uniformly continuous function on $\R$ is
called \emph{limit-periodic} if it is a uniform limit of continuous
periodic functions. But what if we have a uniformly convergent
sequence of discontinuous periodic functions? Can the limit be
limit-periodic? Clearly, we need to assume at least the continuity
of the limit, but what else is needed?

The following statement is likely well known, but since it will play
a role in the proof of our main result, we include its short proof
for the convenience of the reader.

\begin{prop}\label{p.miollification}
Suppose $f \in C(\R)$ is uniformly continuous and, for $n \ge 1$,
$f_n \in L^\infty(\R)$ is periodic. If $\|f_n - f\|_\infty \to 0$ as
$n \to \infty$, then $f$ is limit-periodic.
\end{prop}

\begin{proof}
The issue is that the $f_n$ may be discontinuous and hence the
remedy will be to make them continuous via mollification and then to
observe that the continuous mollified functions still converge
uniformly to $f$. Compare \cite[Section~C.4]{Evans} for the
definitions and general results below.

Explicitly, define $\eta \in C^\infty(\R)$ by
$$
\eta(x) = \begin{cases} C \exp \left( \frac{1}{|x|^2 - 1}\right) & \text{if } |x| < 1, \\ 0 & \text{if } |x| \ge 1, \end{cases}
$$
where $C > 0$ is chosen so that $\int_\R \eta(x) \, dx = 1$. Then, for $\varepsilon > 0$, set
$$
\eta_\varepsilon (x) = \frac{1}{\varepsilon} \eta \left( \frac{x}{\varepsilon} \right)
$$
and, for $n \ge 1$, $f_n^\varepsilon = \eta_\varepsilon \ast f$, that is,
$$
f_n^\varepsilon(x) = \int_\R \eta_\varepsilon(x - y) f_n(y) \, dy.
$$

By the uniform continuity of $f$, Theorem~6 in \cite[Section~C.4]{Evans} and its proof (especially the proof of part (iii)), imply that for each $n \ge 1$ and $\varepsilon > 0$, $f_n^\varepsilon$ is smooth (and in particular continuous) and $\|f_n^\varepsilon - f_n\|_\infty \to 0$ as $\varepsilon \to 0$. Thus, the statement follows by diagonalization, that is, for a suitable sequence $\varepsilon_n \to 0$, the functions $f_n^{\varepsilon_n}$ are continuous, periodic (by construction) and converge uniformly to $f$, showing that $f$ is indeed limit-periodic.
\end{proof}

A function $q$ on $\R$ is called \emph{eventually periodic} if there exists a periodic function $p$ with $p(x) = q(x)$ for all
sufficiently large $x \in \R$.

\begin{coro}\label{c.miollification}
Suppose $f \in C(\R)$ is almost periodic  and, for $n \ge 1$, $f_n \in L^\infty(\R)$ is eventually periodic. If $\|f_n - f\|_\infty \to 0$ as $n \to \infty$, then $f$ is limit-periodic.
\end{coro}

\begin{proof} Let $\varepsilon >0$ be arbitrary. By the preceding proposition it suffices to find a periodic $p \in L^\infty (\R)$ with $\|f - p\|< \varepsilon$.

By assumption there exists an eventually periodic $q \in L^\infty (\R)$ (viz $q = f_m$ for sufficiently large $m$) with
$$
\|f - q\| < \varepsilon.
$$
As $q$ is eventually periodic, there exists a periodic $p \in L^\infty (\R)$ with $p(x) = q (x)$ for all sufficiently large $x$. Let $P>0$ with $p (x) = p (x + P)$ for all $x \in \R$.

As $f$ is almost periodic, there exists a sequence $(t_n)$ in $\R$ with $\|f_{t_n} - f\|_\infty \to 0$ as $n \to \infty$. Here, we set $g_t := g(\cdot - t)$. There exist then unique $k_n\in \N$ and $0 \leq s_n < P$ with $t_n = k_n P + s_n$. Restricting attention to a subsequence if necessary, we can then assume  without loss of generality that $s_n \to s$. As $f$ is uniformly continuous, we can even assume without loss of generality $s_n = s$ for all $n$. To simplify notation we will assume $s =0$.

Hence, $ f - p$ is the pointwise limit of $f - q_{t_n}$ for $n\to \infty$. This gives
\begin{align*}
\|f - p\|_\infty & \leq \limsup_n \|f - q_{t_n}\|_\infty \\
& \leq \limsup_n (\|f - f_{t_n}\|_\infty + \|f_{t_n} - q_{t_n}\|_\infty) \\
& = (\lim_n \|f- f_{t_n}\|_\infty) + \limsup_n \| f_{t_n} - q_{t_n}\|_\infty \\
& = \|f - q\|_\infty\\
& < \varepsilon.
\end{align*}
Here, we used the invariance of $\|\cdot\|_\infty$ under translation in the penultimate step.
\end{proof}

\subsection{Transfer Matrices, Lyapunov Exponents, and Weyl-Titch\-marsh Functions}

This subsection recalls important and well-known concepts, mainly to fix notation.

Fixing $d \ge 2$ and a minimal translation flow on $\T^d$, $\R \ni x \mapsto \omega + x \alpha \in \T^d$, as well as a real-valued sampling function $f \in C(\T^d)$, the \emph{transfer matrices} are defined via
\begin{align*}
\frac{d}{dx} M_f(x,E,\omega) & = A_f(E, \omega + x \alpha) M_f(x,E,\omega) \\
M_f(0,E,\omega) & = I
\end{align*}
for $x \in \R$, $E \in \C$, $\omega \in \T^d$, where
$$
A_f(E, \omega) = \begin{pmatrix} 0 & 1 \\ f(\omega) - E & 0 \end{pmatrix}.
$$
These transfer matrices are defined in such a way that $u$ solves the differential equation
\begin{equation}\label{e.ode}
- u''(x) + f(\omega + x \alpha) u(x) = E u(x)
\end{equation}
if and only if it solves
$$
\begin{pmatrix} u(x) \\ u'(x) \end{pmatrix} = M_f(x,E,\omega) \begin{pmatrix} u(0) \\ u'(0) \end{pmatrix}.
$$

By the subadditive ergodic theorem, there is a number $L(E) \ge 0$, called the \emph{Lyapunov exponent}, so that
$$
L_f(E) = \lim_{|x| \to \infty} \frac{1}{|x|} \log \| M_f(x,E,\omega) \|
$$
for almost every $\omega \in \T^d$.

The map $E \mapsto L_f(E)$ is real-symmetric and subharmonic. Moreover, we have (see \cite[Lemma~3.2 and (49)--(50)]{K97})
\begin{equation}\label{e.lewtmf}
L_f(E) = - \int_{\T^d} \Re m_{+,f,\omega}(E) \, d\omega
\end{equation}
for $E \in \C_+$, the upper half-plane, where $m_{+,f,\omega}$ is the \emph{Weyl-Titchmarsh} $m$-\emph{function} on the right half-line associated with the potential $V(x) = f(\omega + x \alpha)$, defined by
$$
m_{+,f,\omega}(E) = \frac{u_{+,f,\omega}'(0)}{u_{+,f,\omega}(0)},
$$
where $u_{+,f,\omega}$ is a solution of \eqref{e.ode} that is square-integrable at $+\infty$.\footnote{To see that such a solution exists, observe that $E \not\in \sigma(H)$ by self-adjointness, and hence $\tilde u_{+,f,\omega} := (H - E) \chi_{(-1,0)} \in L^2(\R)$. But by definition $\tilde u_{+,f,\omega}$ solves \eqref{e.ode} on $(0,\infty)$. Thus, keeping it unchanged on the right half-line and extending it to a solution on $\R$ by solving \eqref{e.ode}, we obtain $u_{+,f,\omega}$.}

\section{A Semi-Continuity Result}

In this section we discuss the continuum analog of the Avila-Damanik semi-continuity result \cite[Lemma~1]{AD05}. The general structure of the proof will be the same, and hence we will focus mostly on the aspects that are different between the discrete case and the continuum case.

Set
$$
M_R(f) = \mathrm{Leb} \left( \{ E \in \R \cap [-R,R] : L_f(E) = 0 \} \right).
$$

\begin{remark}\label{r.ipksl}
By the Ishii-Kotani-Pastur Theorem \cite[Theorem~4.7]{K97} and the Last-Simon Theorem \cite[Theorem~1.5]{LS99}, we have that $M_R(f) = 0$ if and only if the Schr\"odinger operator in $L^2(\R)$ with potential $V(x) = f(\omega + x \alpha)$ has purely singular spectrum in the energy interval $[-R,R]$ for every $\omega \in \T^d$.
\end{remark}

Here is the continuum analog of \cite[Lemma~1]{AD05}:

\begin{lemma}\label{l.main1}
For all choices of $r, R, \Lambda > 0$, the maps
\begin{equation}\label{e.map1}
(B_r(L^\infty(\T^d)) , \| \cdot \|_1) \to [0,\infty), \quad f \mapsto M_R(f)
\end{equation}
and
\begin{equation}\label{e.map2}
(B_r(L^\infty(\T^d)) , \| \cdot \|_1) \to [0,\infty), \quad f \mapsto \int_0^\Lambda M_R(\lambda f) \, d\lambda
\end{equation}
are upper semi-continuous. Here, $B_r$ denotes the closed ball with radius $r$ in the essential-supremum norm.
\end{lemma}

\begin{proof}
It is enough to show that \eqref{e.map1} is upper semi-continuous, the upper semi-continuity of \eqref{e.map2} then follows from that via Fatou's lemma.

The proof of the upper semi-continuity of \eqref{e.map1} proceeds in the same way as in \cite{AD05}. Assuming that the upper semi-continuity of \eqref{e.map1} fails for some choice of $r, R, \Lambda > 0$, there must be $f_n, f \in L^\infty(\T^d)$ such that
\begin{itemize}

\item[(i)] $f_n \to f$ in $L^1$ and pointwise as $n \to \infty$,

\item[(ii)] $\|f_n\|_\infty \le r$ for every $n \ge 1$ and $\|f\|_\infty \le r$,

\item[(iii)] $\liminf M_R(f_n) \ge M_R(f) + \varepsilon$ for some $\varepsilon > 0$.

\end{itemize}

By (i) and (ii), we have pointwise convergence of the $m$-functions $m_{+,f,\omega}$ in $\C_+$ for almost every $\omega \in \T^d$ (this follows from a modification of the argument given in \cite{JM83}). Thus, by \eqref{e.lewtmf}, (ii), and dominated convergence, the associated Lyapunov exponents $L_{f_n}$ converge pointwise in $\C_+$ to $L_f$.

Next, consider the region $U$ in $\C_+$ bounded by the equilateral triangle $T$ with sides $I, J, K$, where $I = [-R,R] \subset \R$. From here the proof proceeds verbatim as in \cite{AD05}, using the Schwarz-Christoffel formula, as well as the fact that the Lyapunov exponent is harmonic in $\C_+$ and subharmonic (and in particular upper semi-continuous) globally, to derive a contradiction to (iii).
\end{proof}

\section{Small Perturbations That Destroy the Absolutely Continuous Spectrum}

In this section we discuss how arbitrarily small perturbations can destroy the absolutely continuous spectrum. Here, we use results of \cite{KLS11}.

We first recall some basic concepts from \cite{KLS11}. A
\emph{piece} is a pair $(W,I)$ consisting of an interval $I
\subseteq \R$ with length $|I| > 0$ (with $|I| = \infty$ allowed)
and a locally bounded function $W$ on $\R$ supported on $I$. We
abbreviate pieces by $W^I$. Without restriction, we may assume that
$\min I = 0$. A \emph{finite piece} is a piece of finite length. The
\emph{concatenation} $W^I = W_1^{I_1} \mid W_2^{I_2} \mid \ldots$ of
a finite or countable family $(W_j^{I_j})_{j\in N}$, with $N =
\set{1,2,\ldots,N}$ (for $N$ finite) or $N = \N$ (for $N$ infinite),
of finite pieces is defined by
\begin{align*}
I & = \left[0, \sum_{j\in N} |I_j|\right],\\
W & = W_1+\sum_{j\in N,\,j\geq 2}
W_j\Big(\cdot-\Big(\sum_{k=1}^{j-1} |I_k| \Big)\Big).
\end{align*}
In this case we say  that $W^I$ is \emph{decomposed} by $(W_j^{I_j})_{j\in N}$.

Let now  $V$ be a locally bounded function on $\R$. We say that $V$ has the \emph{finite decomposition property} if there exist a finite set $\mathcal{P}$ of finite pieces and $x_0 \in \R$ such that $(1_{[x_0,\infty)} V)$ is a translate of a concatenation $W_1^{I_1} \mid W_2^{I_2} \mid \ldots$ with $W_j^{I_j} \in \mathcal{P}$ for all $j \in \N$. We say that $V$ has the \emph{simple finite decomposition property} if it has
the f.d.p.~with a decomposition such that there is $\ell > 0$ with the following property: Assume that the two pieces
\[
W_{-m}^{I_{-m}} \mid \ldots \mid W_{0}^{I_{0}} \mid W_{1}^{I_{1}} \mid \ldots \mid W_{m_1}^{I_{m_1}} \quad \text{and} \quad W_{-m}^{I_{-m}} \mid \ldots \mid W_{0}^{I_{0}} \mid U_{1}^{J_{1}} \mid \ldots \mid U_{m_2}^{J_{m_2}}
\]
occur in the decomposition of $V$ with a common first part
$W_{-m}^{I_{-m}} \mid \ldots \mid W_{0}^{I_{0}}$ of length at least
$\ell$ and such that
\[
1_{[0,\ell)}(W_{1}^{I_{1}} \mid \ldots \mid W_{m_1}^{I_{m_1}}) = 1_{[0,\ell)}(U_{1}^{J_{1}} \mid \ldots \mid U_{m_2}^{J_{m_2}}),
\]
where $W_j^{I_j}$, $U_k^{J_k}$ are pieces from the decomposition (in particular, all belong to $\mathcal{P}$ and start at $0$) and the latter two concatenations are of lengths at least $\ell$. Then
\[
W_1^{I_1} = U_1^{J_1}.
\]

The relevance of the simple finite decomposition property comes from
the following result from \cite{LSS14} (see \cite{KLS11} as well).

\begin{lemma}[Theorem 7.1 of \cite{LSS14}] Let $W$ be a bounded measurable function on $\R$.
Assume that both $W$ and $W (-\cdot)$ have the simple finite
decomposition property and are not eventually periodic. Then, the
Schr\"odinger operator $H_W\psi = -\psi''(x) + W(x) \psi(x)$ does
not have any absolutely continuous spectrum.
\end{lemma}

Here is the main result of this section.

\begin{prop}\label{p.destroyac}
Given $d \ge 2$, a minimal translation flow $\R \ni x \mapsto \omega
+ x \alpha \in \T^d$, $f \in C(\T^d)$, and $\varepsilon > 0$, there
exists $\tilde f \in L^\infty(\T^d)$ such that $\|f - \tilde
f\|_\infty < \varepsilon$ and, for all $\omega \in \T^d$, the
potential $\tilde V(x) = \tilde f(\omega + x \alpha)$ as well as
$\tilde V(-\cdot)$  have the simple finite decomposition property
and are not eventually  periodic. In particular, the Schr\"odinger
operator in $L^2(\R)$ with potential $\tilde V$ has purely singular
spectrum.
\end{prop}

\begin{proof}
It suffices to show the first statement. The last statement then follows from the preceding lemma.

Since the given flow is minimal, we can assume without loss of generality that the function $f$ yields aperiodic potentials $V(x) = f(\omega + x \alpha)$ (otherwise use a fraction of the given $\varepsilon$ to perturb $f$ within $C(\T^d)$ in order to ensure this property).

For the given $\varepsilon > 0$, let us now consider a sequence of partitions $\mathcal{P}_{\varepsilon,n}$ of $\T^d$ into finitely many boxes (parallelepipeds) of the following form:
$$
B_{\gamma, \mathbf{\ell}} = \left\{ \gamma + \sum_{j = 1}^{d-1} t_j \mathbf{e}_j + t_d \alpha : 0 \le t_j < \ell_j \text{ for } 1 \le j \le d \right\},
$$
where $\gamma \in \T^d$ and $\mathbf{\ell} = (\ell_1, \ldots, \ell_d)$ with $0 < \ell_1, \ldots, \ell_d < 1$. Here $\mathbf{e}_j$ denotes the vector that has a $1$ as its $j$-th component and only $0$'s otherwise.

We require two properties from these partitions. These two properties may be satisfied since $f$ is uniformly continuous and the translation flow is minimal. First we ask that for every $n$ and every box $B_{\gamma, \mathbf{\ell}}$ belonging to $\mathcal{P}_{\varepsilon,n}$, the variation of $f$ on $B_{\gamma, \mathbf{\ell}}$  is less than $\varepsilon/2$, that is,
\begin{equation}\label{e.smallvariation}
\sup_{\omega \in B_{\gamma, \mathbf{\ell}}} f(\omega) - \inf_{\omega \in B_{\gamma, \mathbf{\ell}}} f(\omega) < \frac{\varepsilon}{2}.
\end{equation}
Second, letting $\delta_{\varepsilon,n}$ denote the maximum of $\|\mathbf{\ell}\|_\infty$ taken over all boxes $B_{\gamma, \mathbf{\ell}}$ in the partition $\mathcal{P}_{\varepsilon,n}$, we require that $\delta_{\varepsilon,n} \to 0$ as $n \to \infty$.

Note that once the translation flow enters such a box $B_{\gamma, \mathbf{\ell}}$, then it spends exactly $\ell_d$ time units in the box before it leaves it again. This is true for each entry into the box, no matter where the entry happens.

Let us now define a function $f_{\varepsilon,n} \in L^\infty(\T^d)$ as follows. On each box $B_{\gamma, \mathbf{\ell}}$ belonging to $\mathcal{P}_{\varepsilon,n}$, $f_{\varepsilon,n}$ takes values in the interval
$$
\left[ \inf_{\omega \in B_{\gamma, \mathbf{\ell}}} f(\omega) - \min \left\{ \frac{\varepsilon}{8}, \frac{1}{n} \right\}, \sup_{\omega \in B_{\gamma, \mathbf{\ell}}} f(\omega) + \min \left\{ \frac{\varepsilon}{8}, \frac{1}{n} \right\} \right],
$$
and moreover the value of $f_{\varepsilon,n}$ at the point $\gamma + \sum_{j = 1}^{d-1} t_j \mathbf{e}_j + t_d \alpha$ depends only on $t_d$ and is independent of $t_1, \ldots, t_{d-1}$.\footnote{This will imply the finite decomposition property below.} Finally we require the dependence of $f_{\varepsilon,n}$ on $t_d$ to be continuous and non-constant.\footnote{We can make it even more regular if needed, such as the function taking any value only a finite number of times. This will then imply the simple finite decomposition property.} Such a selection is clearly possible since the interval of allowed values is non-degenerate. Moreover, by construction we have
\begin{equation}\label{e.goodapproximation}
\|f - f_{\varepsilon,n}\|_\infty < \varepsilon.
\end{equation}

Now we claim that there is an $n$ so that the statement of the proposition holds for $\tilde f := f_{\varepsilon,n}$. Assume this fails, and we have that in fact for every $n$, the potential $V_{\varepsilon,n}(x) = f_{\varepsilon,n}(\omega + x \alpha)$ or the potential $V_{\varepsilon,n}(-x)$ is eventually  periodic or does not have the simple finite decomposition property. Now, clearly, these potentials have the finite decomposition property by construction, and the simplicity of the finite decomposition property of the potential follows by \cite[Proposition~3.5]{KLS11}
and the local non-constancy aspect of our construction. Thus, for each $n$ the potential $V_{\varepsilon,n}$ or  $V_{\varepsilon,n}
(-\cdot)$ must be eventually periodic. Restricting attention to a subsequence we  can assume without loss of generality that
$V_{\varepsilon,n}(x)$ must be eventually periodic for every $n$. Note that the $V_{\varepsilon,n}$ are bounded and measurable, but in
general discontinuous. These eventually periodic functions converge (by construction) uniformly to the function $V(x) = f(\omega + x \alpha)$, which is clearly almost periodic, and hence must be limit-periodic due to Corollary~\ref{c.miollification}. But since it is manifestly quasi-periodic as well, it must therefore be periodic by \cite[Corollary~A.1.4]{AS81}; contradiction (by our initial step).
\end{proof}

\begin{remark}\label{r.addinglambda}
In the proposition above, once we know that the potential $\tilde V(x) = \tilde f(\omega + x \alpha)$ and $\tilde V (- \cdot)$ have the simple finite decomposition property and are not eventually periodic, these properties are inherited by any non-zero multiple of the potential. In particular it then also follows that, for every $\lambda > 0$, the Schr\"odinger operator in $L^2(\R)$ with potential $\lambda \tilde V$ has purely singular spectrum.
\end{remark}

\section{Closing the Jumps}

We saw in Proposition~\ref{p.destroyac} that by approximating a given continuous sampling function with a discontinuous sampling function, we can destroy the absolutely continuous spectrum of the associated operator. The approximation is with respect to the $\|\cdot\|_\infty$ norm. However, we wish to identify a \emph{continuous} sampling function that is close to the original one, for which the absolutely continuous spectrum is empty. A second approximation is therefore necessary to \emph{close the jumps}.

Clearly, the discontinuous function (with the desired property) cannot be approximated by a continuous function in the $\|\cdot\|_\infty$ norm. However, it is possible to approximate it in the $\|\cdot\|_1$ norm. This shows why the semi-continuity result given by Lemma~\ref{l.main1} is relevant. Moreover, since the limit function has a zero value and the values are non-negative, the semi-continuity result becomes in effect a continuity result in the setting relevant to this discussion.

The following lemma implements this two-step approximation:

\begin{lemma}\label{l.main2}
For $f \in C(\T^d)$ and $0 < \varepsilon, \delta, R, \Lambda < \infty$, there exists $g \in C(\T^d)$ such that $\|f - g\|_\infty < \varepsilon$, $M_R(g) < \delta$, and $\int_0^\Lambda M_R(\lambda g) \, d \lambda < \delta$.
\end{lemma}

\begin{proof}
Given $f \in C(\T^d)$ and $0 < \varepsilon, \delta, R, \Lambda < \infty$, Proposition~\ref{p.destroyac} yields an $\tilde f \in L^\infty(\T^d)$ with $\|f - \tilde f\|_\infty < \frac{\varepsilon}{2}$ and $M_R(\tilde f) = 0$, as well as (cf.~Remark~\ref{r.addinglambda}) $M(\lambda \tilde f) = 0$ for every $\lambda > 0$.

Let us mollify $\tilde f$ (via the mollifiers used in the proof of Proposition~\ref{p.miollification}) to produce $f_n \in C(\T^d)$ with
$$
\lim_{n \to \infty} \|f_n - \tilde f\|_1 = 0
$$
and
$$
\sup_{n \in \Z_+} \|f_n - f\|_\infty < \varepsilon.
$$
By the non-negativity of the quantities in question, the vanishing limits, and the semi-continuity properties from Lemma~\ref{l.main1}, it follows that
$$
\lim_{n \to \infty} M_R(f_n) = 0
$$
and
$$
\lim_{n \to \infty} \int_0^\Lambda M_R(\lambda f_n) \, d\lambda = 0.
$$
Thus, for $n$ large enough, $g = f_n$ has the desired properties.
\end{proof}

\section{Proof of the Main Results}

In this section we prove the main results, Theorems~\ref{t.main1} and \ref{t.main2}. The proofs are analogous to the corresponding proofs in \cite{AD05}. Since they are very short, we give the details for the reader's convenience.

\begin{proof}[Proof of Theorem~\ref{t.main1}.]
For $0 < \delta, R < \infty$, we define
$$
M_{R,\delta} = \{ f \in C(\T^d) : M_R(f) < \delta \}.
$$
By Lemma~\ref{l.main1}, $M_{R,\delta}$ is open, and by Lemma~\ref{l.main2}, $M_{R,\delta}$ is dense. Thus,
$$
\{ f \in C(\T^d) : \Sigma_\mathrm{ac}(f) = \emptyset \} = \bigcap_{n \in \Z_+} M_{n,\frac1n}
$$
is a dense $G_\delta$ set, as claimed.
\end{proof}

\begin{proof}[Proof of Theorem~\ref{t.main2}.]
For $0 < \delta, R , \Lambda < \infty$, we define
$$
M_{R,\delta}(\Lambda) = \left\{ f \in C(\T^d) : \int_0^\Lambda M_R(\lambda f) \, d \lambda < \delta \right\}.
$$
By Lemma~\ref{l.main1}, $M_{R,\delta}(\Lambda)$ is open, and by Lemma~\ref{l.main2}, $M_{R,\delta}(\Lambda)$ is dense. Thus,
$$
\{ f \in C(\T^d) : \Sigma_\mathrm{ac}(\lambda f) = \emptyset \text{ for a.e. } \lambda > 0 \} = \bigcap_{n \in \Z_+} M_{n,\frac1n}(n)
$$
is a dense $G_\delta$-set, as claimed.
\end{proof}

\end{document}